\documentclass [12pt]{amsart}

\usepackage[left=1in,right=1in,top=1in,bottom=1in]{geometry} 

\usepackage{amsmath}

\usepackage{url}

\usepackage[normalem]{ulem}

\usepackage{amssymb,amsthm,amsmath}

\usepackage{ifthen}

\usepackage{graphicx}

\DeclareGraphicsExtensions{.eps}

\newtheorem{theorem}{Theorem}[section]

\newtheorem{lemma}[theorem]{Lemma}

\newtheorem{cor}[theorem]{Corollary}

\theoremstyle{definition}

\newtheorem{definition}[theorem]{Definition}

\newtheorem{example}[theorem]{Example}

\newtheorem{remark}[theorem]{Remark}

\numberwithin{equation}{section}

\begin{document}

\date{\today}
\title{Banach Spaces of GLT Sequences and Function Spaces}
\author{V. B. Kiran Kumar}
 \email{vbk@cusat.ac.in} 
\author{ Rahul Rajan}
 \email{rajan.rahul48@yahoo.com} 
 \author{ N. S. Sarath Kumar}
 \email{sarathkumarns05@gmail.com} 
\address{Department of Mathematics, Cochin University of Science And Technology, Kerala, India.}

\begin{abstract}
       The Generalized Locally Toeplitz (GLT) sequences of matrices have been originated from the study of certain partial differential equations. To be more precise, such matrix sequences arise when we numerically approximate some partial differential equations by discretization. The study of the asymptotic spectral behaviour of GLT sequence is very important in analysing the solution of corresponding partial differential equations. The approximating classes of sequences (a.c.s) and the spectral symbols are important notions in this connection. Recently, G. Barbarino obtained some additional results regarding the theoretical aspects of such notions. He obtained the completeness of the space of matrix sequences with respect to pseudo metric a.c.s. Also, he identified the space of GLT sequences with the space of measurable functions. In this article, we follow the same research line and obtain various results connecting the sub-algebras of matrix sequence spaces and sub-algebras of function spaces. In some cases, these are  identifications as Banach spaces and some of them are Banach algebra identifications. In the process, we also prove that the convergence notions in the sense of eigenvalue/singular value clustering are  equivalent to the convergence with respect to the metrics introduced here. These convergence notions are related to the study of preconditioners in the case of matrix/operator sequences. Finally, as an application of our main results, we establish a Korovkin-type result in the setting of GLT sequences.
\end{abstract}
\maketitle

	\section{Introduction and Preliminaries}
	The correspondence between matrix sequences and measurable functions is very natural in many important examples such as the Toeplitz matrices. Here the spectral information of the operator/matrix sequence is stored in the corresponding symbol-function (recall the celebrated Szeg\"o distribution theorem \cite{szego}). 
	Also, such matrix sequences arise naturally in the study of partial differential equations with certain boundary conditions, using the finite difference approximation.
	
	For example, if we consider the Schr\"odinger operator that maps $f\mapsto -f^{''}+v.f$, where $v$ is a real-valued periodic potential function, then the corresponding finite difference approximation leads to a sequence of block Toeplitz matrices. If we consider more general PDE's like those arise from diffusion problem ($f\mapsto (-af^{'})^{'}+v.f$) or convection-diffusion-reaction  ($f\mapsto (-af^{'})^{'}+bf^{'}+v.f$), we end up with Locally Toeplitz (LT) or Generalized Locally Toeplitz (GLT) sequences \cite{stefano}. In most of the cases, we see that the sequence of discretization matrices $\{A_n\}_n$ enjoys an asymptotic spectral distribution. This is somehow related to
	the spectrum of the differential operator associated with the considered PDE. Asymptotic singular value distribution is defined below
	\begin{definition}
		We say that $\{A_n\}_n$ has an asymptotic singular value distribution with symbol $f$ and  write
		$\{A_n\}_n\sim_{\sigma} f$ , if, for all $F\in C_C(\mathbb{R})$
		$$ \lim _{n\to\infty} \frac{1}{n}\sum_{i=1}^{n}F(\sigma_i(A_n)) = \frac{1}{2\pi}\int_D F(|f(x)|) dx. $$
	\end{definition} 
	
	Because of this inherent connection between the matrix sequences and the corresponding symbol-functions, many researchers explored the possible generalizations of such results. They tried to obtain symbol-functions or function spaces corresponding to some class of matrix sequences to understand the spectral asymptotic. Recently, such studies have been initiated in the setting of GLT sequences by researchers like Albrecht B\"ottcher, Stefano Serra-Capizzano, G. Barbarino, C. Garoni, etc. \cite{barbarino,garoni,stefano1,bottcher,stefano}.\\
	An equivalence between GLT sequences and measurable functions was obtained in \cite{barbarino}. In this article, we follow  the same research line and obtain various results connecting the sub-algebras of the space of all matrix sequences and the sub-algebras of  space of measurable functions. In some cases, these are  identifications as Banach spaces and some of them are Banach algebra identifications. These spaces of matrix sequences are defined using various pseudo-metric functions introduced in this article. These notions are motivated from the  pseudo-metric induced by the notion of approximating class of sequences (a.c.s) used in \cite{stefano}. As in \cite{stefano}, we also obtain characterizations  of  convergence notions in the sense of eigenvalue/singular value clustering here (these notions were originated from  the preconditioning problems in numerical linear algebra).\\
	Now we list down some definitions which are useful throughout this article.
	\begin{definition}
		Let $\lbrace A_{n}\rbrace _{n} $ be a matrix-sequence and $\lbrace\lbrace B_{n,m} \rbrace_n\rbrace_m$ a sequence of matrix-sequences. We say that $\lbrace\lbrace B_{n,m} \rbrace_n\rbrace_m$ is an \textit{approximating class of sequences} ${(a.c.s)}$ for  $\lbrace A_{n}\rbrace _{n}$ if the following condition is met: for every $m$ there exists an $n_m$ such that, for $n \geq n_m$,
		\begin{equation*}\label{eu_eqn}
			A_n= B_{n,m}+R_{n,m}+N_{n,m}
		\end{equation*}
		$\text{rank} (R_{n,m}) \leq c(m)n,\quad \| N_{n,m} \| \leq \omega(m) $ 
		where $\parallel.\parallel$ is the spectral norm,  $ n_m, c(m)$ and $\omega(m)$ depend only on $m$ and 
		$$\lim_{m\to\infty} c(m) = \lim_{m\to\infty} \omega(m)=0.$$
	\end{definition}	
	The notion of approximating classes of sequences  is a powerful tool in numerical Linear Algebra literature. Using this , we can replace a complicated matrix sequence $\{A_n\}$ by some simpler matrix sequences $\{\{B_{n,m}\}_n\}_m$. The asymptotic distribution of singular values/eigenvalues of $\{\{B_{n,m}\}_n\}_m$ can be used to compute the asymptotic distribution of singular values/eigenvalues of $\{A_n\}_n$.
	
	\begin{definition}\label{GLT defn}
		\normalfont Let $ a:[0,1]\rightarrow \mathbb{C}$ be a Riemann-integrable function and $f \in L^1{([-\pi,\pi])}$. We say that a matrix sequence $\lbrace A_{n}\rbrace _{n}$ is a \textit{Locally Toeplitz (LT) sequence} with symbol $a \otimes f$, and we write $\lbrace A_{n}\rbrace _{n}\sim_{\mathrm{LT}} a\otimes f$, if $\lbrace \lbrace LT_{n}^{m}(a,f) \rbrace_n\rbrace_{m \in \mathbb{N}}$ is an a.c.s for $\lbrace A_{n}\rbrace _{n}$, where
		\begin{equation*}
			\begin{aligned}
				LT_{n}^{m}(a,f)&=[D_m(a) \otimes T_{\lfloor n/m \rfloor}(f)]\oplus O_{n (mod\, m)}\\
				&=\underset{i=1,...,n}{\mathrm{diag}}[a(\frac{i}{m})T_{\lfloor n/m \rfloor }(f)]\oplus O_{n (mod\, m)}.
			\end{aligned}
		\end{equation*}
		Here $D_n(a)$  is a  $(n \times n)$ diagonal matrix associated with $a$ given by
		$$D_n(a)=\underset{i=1,...,n}{\mathrm{diag}} a(\frac{i}{n}).$$
		Let  $\kappa:[0,1] \times [-\pi,\pi] \rightarrow \mathbb{C}$ be a measurable function. We say that a matrix sequence $\lbrace A_{n}\rbrace _{n}$  is a \textit{Generalized  Locally Toeplitz (GLT) sequence} with symbol $\kappa$, and we write $\lbrace A_{n}\rbrace _{n}\sim_{\mathrm{GLT}} \kappa$, if the following condition is met.\\
		For every $m$ varying in some infinite subset of $\mathbb{N} $ there exists a  finite number of LT sequence $\lbrace A_{n}^{(i,m)}\rbrace _{n} \sim_{\mathrm{LT}} a_{i,m}\otimes f_{i,m},i=1,...,k_m ,$ such that:
		\begin{itemize}
			\item $\sum \limits_{i=1}^{k_m} a_{i,m} \otimes f_{i,m}\rightarrow \kappa $ in measure over $[0,1] \times [-\pi,\pi]$ when $m \rightarrow  \infty$;
			\item $\lbrace\lbrace \sum \limits_{i=1}^{k_m} A_{n}^{(i,m)}\rbrace _{n}\rbrace_m $ is an a.c.s for $\lbrace A_{n}\rbrace _{n}$.
		\end{itemize}
	\end{definition}
	In this article, the results presented are for $D=[0,1]\times[-\pi,\pi]$. All these results are directly follows for the multilevel case, $D=[0,1]^d\times[-\pi,\pi]^d$.
	
	The major outcome of this article is the identification of a some subspace, $\tilde{\mathcal{G}}^p$ of  GLT sequences (we will introduce this space later) with subspaces of the measurable functions. We state the result below.
	\begin{theorem}\label{isometry}
		Let $D=[0,1]\times[-\pi,\pi]$. The Banach spaces $\tilde{\mathcal{G}}^p$ and $L^p(D)$, $D=[0,1]\times[-\pi,\pi]$, $1\leq p \leq \infty $ are isometrically isomorphic. In particular  $\tilde{\mathcal{G}}^{\infty}$ and $L^\infty(D)$ isomorphic as $C^*$-algebra. (The proof will be given in section 3)
	\end{theorem}
	
	The completeness of $\tilde{\mathcal{E}}$, the set of all equivalence class of matrices of increasing order with respect to a.c.s metric  was proved in \cite{barbarino}. Also it is known that the class of GLT sequences with respect to a.c.s metric form a complete *- algebra \cite{stefano}. This space is isometrically  isomorphic to the class of measurable functions on $D$ \cite{barbarino}.

	In this article, we introduce seminorms  $q_w$ and $q_{w^p}$ in connection with these convergence notions and obtain the Banach spaces $\tilde{\mathcal{A}}_w$ and $\tilde{\mathcal{A}}_{w^p}$ with respect to the norms induced by these seminorms . In particular  we prove that $\tilde{\mathcal{A}}_w$ is a $C^*$-algebra. The space $\tilde{\mathcal{G}}^p$ is the collection of GLT sequences in $\tilde{\mathcal{A}}_{w^p}$, for $1\leq p\leq \infty$. It turns out that $\tilde{\mathcal{G}}^p$ are Banach spaces with respect to the norms induced by these seminorms. Our main  result stated above says that  $\tilde{\mathcal{G}}^p$ are isometrically isomorphic to $L^p(D)$. 
	
	The article is organized as follows. In the next section, we introduce the seminorms $q_w$, $q_{w^p}$  and obtain its relation with Type 2 weak cluster convergence. Also, we identify the Banach spaces $\tilde{\mathcal{A}}_w$ and $\tilde{\mathcal{A}}_{w^p}$ of matrix sequences. In the third section, we prove our main results; obtain the equivalence between these Banach spaces with subspaces of measurable functions. In the fourth section, as an application of our main results, we obtain a Korovkin-type approximation theorem for GLT sequences analogous to the result for Toepliz sequences. The article ends with a concluding section, mentioning some further possibilities.
	\section{Banach Spaces of Matrix Sequences}
	
	Motivated from the notion of a.c.s, we introduced certain seminorms on  the space of all sequences of matrices of increasing size.\\
	Let $\mathcal{E}=\{\{A_{n}\}_n: A_{n} \text{ is a matrix of finite order}\}.$
	For $A_n\in M_n(\mathbb{C})$, let
	$$P(A_n)=inf\left\{\frac{rank(R_n)}{n}+\|{N_n}\|\,:\, R_n+N_n=A_n, R_n,N_n\in M_n(\mathbb{C})\right\},$$
	where infimum is taken over all decomposition of $A_n=R_n+N_n$.\\
	Let $\{A_n\}_n\in \mathcal{E}$, we can define 
	$$p(\{A_n\}_n)=\limsup_{n\to \infty}P(A_n).$$
	For $\{A_n\}_n,\{B_n\}_n\in \mathcal{E}$, define
	$$d_{acs}(\{A_n\}_n,\{ B_n\}_n)=p(\{A_n-B_n\}_n).$$
	
	It was proved in \cite{barbarino,garoni1} that $d_{acs}$ is a pseudo metric on $\mathcal{E}$ which turns $\mathcal{E}$ into a complete pseudometric space $(\mathcal{E},d_{acs})$ and the convergence of $\{ \{B_{n,m}\}_{n} \}_{m}$  to $\{A_n\}_{n}$ is called  a.c.s. convergence, denoted by $\{\{B_{n,m}\}_n\}_m\xrightarrow{a.c.s.}\{A_n\}_n$ as $m\to \infty$.  	
	Let $L=\{ \{A_n\}_n\in \mathcal{E}: p(\{A_n\}_n)=0\}$. Then the quotient space $\tilde{\mathcal{E}}=\mathcal{E}/L$ will be a metric space with respect to the metric $\tilde{d}_{acs} :\tilde{\mathcal{E}} \times \tilde{\mathcal{E}} \to \mathbb{R}$ defined by $$\tilde{d}_{acs}(\{A_n\}_n+L,\{B_n\}_n+L)=d_{acs}(\{A_n\}_n,\{B_n\}_n)$$
	The following Theorem in \cite{garoni1} gives an equivalent definition for $P(A)$.
	\begin{theorem}(Theorem 5 of \cite{garoni1})\label{p(a)}
		For any matrix $A_n\in M_n(\mathbb{C})$,
		$$P(A_n)=\min_{i=1,2,\ldots,n}\left\{\frac{i}{n}+\sigma_{i+1}(A_n)\right\},$$
		where $\sigma_i(A_n)$ is the $i^{th}$ singular value of $A_n$ arranged in non-increasing order and we assume by convention that $\sigma_{n+1}(A_n)=0.$
	\end{theorem}
	\begin{remark}
		$\tilde{d}_{acs}$ in $\tilde{\mathcal{E}}$ is not induced from any norm. For if, $\{A_n\}_n=\{I_n\}_n$, sequence of identity matrices and $\{B_n\}_n=\{O_n\}_n$ is the sequence of zero matrices, then
		\begin{equation*}
			\begin{aligned}
				d_{acs}(\{A_n\}_n,\{B_n\}_n)&=d_{acs}(\{I_n\}_n,\{O_n\}_n)&=1.\\
				d_{acs}(\{2A_n\}_n,\{2B_n\}_n)&=d_{acs}(\{2I_n\}_n,\{O_n\}_n)&=1.
			\end{aligned}
		\end{equation*}
		Therefore $d_{acs}(\{2A_n\}_n,\{2B_n\}_n)\neq 2d_{acs}(\{A_n\}_n,\{B_n\}_n)$.
	\end{remark}
	\begin{definition}\normalfont
		Let $\{A_n\}_n$ be a matrix sequence and the functions  $q_w,q_{w^p}:\mathcal{E}\to \mathbb{R}$ defined as
				$$q_w(\{A_n\}_n)=\inf \left\{\limsup_{n\to \infty}\|N_n\|:\,\, R_n+N_n=A_n, \,\, \lim_{n\to \infty}\frac{rank R_n}{n}=0\right\},$$
				$$q_{w^p}(\{A_n\}_n)=\inf \left\{\limsup_{n\to \infty}\frac{\|N_n\|_{sp}}{n^{1/p}}:\,\, R_n+N_n=A_n, \,\, \lim_{n\to \infty}\frac{rank R_n}{n}=0\right\},
				1\leq p< \infty.$$
		Here the infimum is taken over all such decompositions of $A_n$ and $\|N_n\|_{sp}$ denotes schatten $p-$norm.
	\end{definition}
	The  subspaces $\mathcal{A}_w$ and $\mathcal{A}_{w^p}$ of $\mathcal{E}$ are defined as follows: 
	$$\mathcal{A}_w=\left\{\{A_n\}_n\in \mathcal{E}\,\,:\,\,q_w(\{A_n\}_n)<\infty\right\}\,,\quad 
	\mathcal{A}_{w^p}=\left\{\{A_n\}_n\in \mathcal{E}\,\,:\,\,q_{w^p}(\{A_n\}_n)<\infty\right\}.$$
	
	Now we recall the notion of strong cluster convergence, weak cluster convergence and uniform cluster convergence used in \cite{rahul,stefano3}. 
	\begin{definition}\normalfont
		Let $\{A_n\}_n $ and $ \{B_n\}_n$ be two sequences of matrices of increasing size. We say that $ \{A_n-B_n\}_n$ converges to constant sequence $\{O_n\}_n$ (sequence of zero-matrices) in \textit{Type 2 weak  cluster sense} if for any $\epsilon > 0$, there exists integers $n_{1,\epsilon} ,n_{2,\epsilon}$, such that for $n>n_{2,\epsilon}, A_n-B_n=R_n+N_n$ with rank $R_n\leq n_{1,\epsilon}$ and $ \|N_n\|<\epsilon$. Also $n_{1,\epsilon}$ depends on both $n$ and $\epsilon$ and is of $o(n)$. \\	
		The convergence is in the \textit{ Type 2 uniform cluster sense} if $n_{1,\epsilon}$ is independent of $\epsilon$ and in the \textit{ Type 2 strong cluster sense } if $n_{1,\epsilon}$ depends only on  $\epsilon$.
	\end{definition}
	\begin{remark}(\cite{rahul})\label{sigma clustering}
		$\{A_n-B_n\}_n$ converges to $\{O_n\}_n$ in Type 2 weak cluster sense if and only if for any $\epsilon >0$, there exist integers $n_{1,\epsilon},n_{2,\epsilon}$ such that for all $n_{2,\epsilon}$, except at most possibly $n_{1,\epsilon}$ (dependent of size $n$ and is of $o(n)$) singular values, all singular values $\{A_n-B_n\}_n$ lie in the interval $[0,\epsilon)$. The Type 2 convergence is equivalent to the singular value clustering. There is a notion of Type 1  convergence that is equivalent to eigenvalue clustering. Both originated from the study of preconditioners in numerical linear algebra problems (see \cite{stefano3} for eg.).
	\end{remark}
	The following lemma is a consequence of the results in \cite{tyrtyshnikov1996unifying} which provides a
	criterion to establish the convergence notions defined above.
	\begin{lemma}\label{frobenius}
		Let $\{A_n\}_n$ and $\{B_n\}_n$ be two sequences of $n\times n$  matrices of growing order. If $\|A_n-B_n\|_{F}^2=o(n)$ , then we have the convergence in the Type 2 weak cluster sense.	If $\|A_n-B_n\|_{F}^2=O(1)$, then the convergence is Type 2 strong cluster sense.
	\end{lemma}
	In \cite{stefano} Carlo Garoni and Stefano Serra-Capizzano proved that a.c.s convergence and Type 2 weak convergence are equivalent. We state the result below.
	\begin{theorem}(Theorem 4.1 of \cite{stefano})\label{thm1}
		Let $\{A_n\}_n $ and $ \{B_n\}_n$ be two sequences of matrices of increasing size. Then $\{A_n-B_n\}_n$ converges to $\{O_n\}_n$ in type 2 weak cluster sense if and only if $d_{acs}(\{A_n\}_n,\{B_n\}_n)=0.$ 
	\end{theorem}
	Next theorem gives a characterization for $q_w$ leads to a relation between type 2 weak convergence and  $q_w$, analogous to Theorem \ref{thm1}.
	\begin{theorem}\label{sigma}
		Let $\{A_n\}_n$ be a matrix sequence and $\sigma_i(A_n)$ be the $i^{th}$ singular value of the matrix $A_n$  arranged in non-increasing order. Then,
		$$q_w(\{A_n\}_n)=\inf\left\{\alpha \in [0,\infty):\lim_{n\to \infty}\frac{\#(\sigma(A_n)\geq \alpha )}{n}=0\right\}.$$	
	\end{theorem}
	\begin{proof}
		
		Let $\{R_n\}_n$ and $\{N_n\}_n$ be any matrix sequence such that $\{R_n\}_n+\{N_n\}_n=\{A_n\}_n$ and $\lim_{n\to \infty}\frac{rank\,R_n}{n}=0$.\\
		Let $\sigma_1(A_n)\geq\sigma_2(A_n)\geq\ldots\geq \sigma_n(A_n)$ be the singular values of $A_n$ arranged in non increasing order. We know that
		\begin{equation*}
			\begin{aligned}
				\sigma_i(A_n)\leq\sigma_i(R_n)+\|N_n\|.
			\end{aligned}
		\end{equation*}
		Setting $r_1=rank\,R_n$, then $\sigma_i(A_n)\leq  \|N_n\|$, for all $i>r_1$. Let $r_2$ be the smallest integer such that $\sigma_i(A_n) \leq  \|N_n\|$ for all $i>r_2$.  
		
		Then $r_1\geq r_2$, and $\sigma_{r_2}(A_n)>\|N_n\|\geq\sigma_{r_2+1}(A_n)$. Let $A_n=U_n\Sigma_n{V}_{n}^{*}$ be a Singular Value Decomposition (SVD) of $A_n$ and set,
		$$\tilde{R}_n=U_n\,diag(\sigma_1(A_n),\ldots,\sigma_{r_2}(A_n),0,\ldots,0){V}_{n}^{*}$$
		$$\tilde{N}_n=U_n\,diag(0,\ldots,0,\sigma_{r_2+1}(A_n),\ldots,\sigma_n(A_n)){V}_{n}^{*}.$$  
		Then, $A_n=\tilde{R}_n+\tilde{N}_n$, let $\limsup_{n\to \infty}\|\tilde{N_n}\|=\alpha$.
		$$rank(R_n)=r_1\geq r_2=rank(\tilde{R}_n),\,\,\|N_n\|\geq\sigma_{r_2+1}(A_n)=\|\tilde{N}_n\|.$$
		Then, $\displaystyle \lim_{n\to \infty}\frac{\#(\sigma(A_n)>\alpha)}{n}=\lim_{n\to \infty}\frac{r_2}{n}\leq \lim_{n\to \infty}\frac{r_1}{n}=0 \text{ and}$ $$\limsup_{n\to \infty}\|N_n\|\geq\limsup_{n\to \infty}\|\tilde{N}_n\|=\alpha.$$
		Therefore, $q_w(\{A_n\}_n)\geq\inf\left\{\alpha \in [0,\infty):\lim_{n\to \infty}\frac{\#(\sigma(A_n)\geq \alpha )}{n}=0\right\}.$
		
		To prove the other inequality, let $A_n=U_n\Sigma_n{V}_{n}^{*}$ be a SVD of $A_n$. Let $\alpha \in [0,\infty)$ such that $\lim_{n\to \infty}\frac{\#(\sigma(A_n)\geq \alpha )}{n}=0$. Let $R_n=U_n\tilde{\Sigma}_nV_{n}^{*}, N_n=U_n\hat{\Sigma}_nV_{n}^{*}$, where $\tilde{\Sigma}_n$ is the diagonal  matrix obtained from $\Sigma_n$ by setting $0$ to all the singular values of $A_n$ that are less than or equal to $\alpha$, and $\hat{\Sigma}_n=\Sigma-\tilde{\Sigma}_n$. Hence $\lim_{n\to \infty}\frac{rank(\tilde{\Sigma}_n)}{n}=0$.
		Then, ${A_n}={R_n}+{N_n},\, rank(R_n)=\#\{\sigma(A_n)\geq \alpha \}$ and $\limsup_{n\to \infty}\|N_n\|\leq \alpha$. By taking infimum over all such $\alpha,$ we get
		\begin{equation*}
			\begin{aligned}
				q_w(\{A_n\}_n)&\leq \inf\left\{\alpha \in [0,\infty) : \limsup_{n\to \infty}\|N_n\|\leq \alpha\right\}\\
				&\leq\inf\left\{\alpha \in {[0,\infty)}:\lim_{n\to \infty}\frac{\#(\sigma(A_n)\geq \alpha )}{n}=0\right\}
			\end{aligned}
		\end{equation*}		
	\end{proof}	
	Using the characterization of $q_w$ and remark \ref{sigma clustering}, we obtain the following corollaries.
	\begin{cor}\label{qw}
		Let $\{A_n\}_n $ and $ \{B_n\}_n$ be two sequences of matrices of increasing size. Then $\{A_n-B_n\}_n$ converges to $\{O_n\}_n$ in type 2 weak cluster sense if and only if $q_w(\{A_n-B_n\}_n)=0.$
	\end{cor}
	\begin{cor}
		Let $\{A_n\}_n $ and $ \{B_n\}_n$ be two sequences of matrices of increasing size. Then $\{A_n-B_n\}_n$ converges to $\{O_n\}_n$ in type 2 weak cluster sense if and only if $q_{w^p}(\{A_n-B_n\}_n)=0.$
	\end{cor}
	\begin{proof}
		The result follows from $\frac{\|N_n\|_{sp}}{n^{1/p}}\leq \|N_n\|$ and the Corollary \ref{qw}.
	\end{proof}
	Let  $L_w=\{ \{A_n\}_n\in\mathcal{A}_w: q_w(\{A_n\}_n)=0\},\, L_{w^p}=\{ \{A_n\}_n\in\mathcal{A}_{w^p}: q_{w^p}(\{A_n\}_n)=0\}$. Then  $\tilde{\mathcal{A}}_{w}=\mathcal{A}_w/L_w$ and $\tilde{\mathcal{A}}_{w^p}=\mathcal{A}_{w^p}/L_w$ are the quotient spaces of $\mathcal{A}_{w}$ and $\mathcal{A}_{w^p}$ respectively.
	
	Now we prove that $\tilde{\mathcal{A}}_{w}$ and $\tilde{\mathcal{A}}_{w^p}$ are Banach spaces.
	\begin{theorem}\label{completeness}
		$\tilde{\mathcal{A}}_w$ and $\tilde{\mathcal{A}}_
		{w^p}$  are Banach spaces with respect to the norms induced by $q_w$ and $q_{w^p}$ respectively. In particular, $\tilde{\mathcal{A}}_w$ forms a  $C^*$-algebra and $\tilde{\mathcal{A}}_{w^2}$ is a Hilbert space.
	\end{theorem} 
	\begin{proof}
		\sloppy Here we prove only the case of $\tilde{\mathcal{A}}_w$ and  the other case is similar. First we fix some notations. Let $q_A= q_w(\{A_n\}_n), q_B = q_w(\{B_n\}_n), q_{A+B}=q_w(\{A_n\}_n+\{B_n\}_n)$ and $q_{AB}=q_w(\{A_nB_n\}_n)$. From the definition of $q_w$, for every $m\in\mathbb{N}$, there exist four matrix  sequences $\{R_{n,m}^{A}\}$,$\{N_{n,m}^{A}\}$,
		$\{R_{n,m}^{B}\}$,$\{N_{n,m}^{B}\}$ such that
		$\{R_{n,m}^{A}\}+\{N_{n,m}^{A}\}=\{A_n\}_n,\,\,\, \{R_{n,m}^{B}\}+\{N_{n,m}^{B}\}=\{B_n\}_n$, and
		$$\limsup_{n\to \infty}\|N_{n,m}^A\|\leq q_A+\frac{1}{m},\,\,\,\limsup_{n\to \infty}\|N_{n,m}^B\|\leq q_B+\frac{1}{m}.$$
		Also, $\displaystyle \lim_{n\to \infty}\frac{rank (R_{n,m}^{A})}{n}=\lim_{n\to \infty}\frac{rank (R_{n,m}^{A})}{n}=0$.\\
		Now we verify the axioms of seminorm, the non negativity and $q_w(\{O_n\}_n)=0$ are trivial. For triangular inequality,  
		\begin{equation*}
			\begin{aligned}
				q_{A+B}\,\,&\leq \limsup_{n\to \infty}\|N_{n,m}^{A}+N_{n,m}^{B}\|\\
				&\leq \limsup_{n\to \infty}\|N_{n,m}^{A}\|+\limsup_{n\to \infty}\|N_{n,m}^B\|\\
				&\leq q_A+q_B+\frac{2}{m}.
			\end{aligned}
		\end{equation*}
		Thus, $q_{A+B}\leq q_A+q_B$.\\
		$q_{\alpha A}= |\alpha|q_{A}, \alpha \in \mathbb{C}$ is straight forward. 
		
		Hence $q_w$ is a seminorm in $\mathcal{A}_w$. Then the function  $\tilde{q}_w : \tilde{\mathcal{A}}_w \times \tilde{\mathcal{A}}_w \to \mathbb{R}$  defined as $$\tilde{q}_w(\{A_n\}_n+L_w)=q_w(\{A_n\}_n),$$
		becomes a norm on $\tilde{\mathcal{A}}_w$.\\
		For the Banach algebra inequality, we consider
		$$\{A_nB_n\}_n =\{R_{n,m}^{A}R_{n,m}^{B}+R_{n,m}^{A}N_{n,m}^{B}+N_{n,m}^{A}R_{n,m}^{B}\}_n+\{N_{n,m}^{A}N_{n,m}^{B}\}_n.$$
		Here $\lim_{n\to \infty}\frac{1}{n} rank(R_{n,m}^{A}R_{n,m}^{B}+R_{n,m}^{A}N_{n,m}^{B}+N_{n,m}^{A}R_{n,m}^{B})= 0.$ 
		Then,
		\begin{equation*}
			\begin{aligned}
				q_w(\{A_nB_n\}_n)&\leq \limsup_{n\to \infty}\|N_{n,m}^{A}N_{n,m}^{B}\|\\
				&\leq \limsup_{n\to \infty}\|N_{n,m}^{A}\|\limsup_{n\to \infty}\|N_{n,m}^B\|\\
				&\leq (q_w(\{A_n\}_n)+\frac{1}{m})(q_w(\{B_n\}_n)+\frac{1}{m}).
			\end{aligned}
		\end{equation*}
		Thus, $q_w(\{A_nB_n\}_n)\leq q_w(\{A_n\}_n)q_w(\{B_n\}_n)$ ( Note that this result is not true for $\tilde{\mathcal{A}}_{w^p}$).
		
		Finally, we  prove the completeness of $\tilde{\mathcal{A}}_w$. 
		Let $\{\{B_{n,m}\}_n+L_w\}_m$ be a Cauchy sequence in $\tilde{\mathcal{A}}_w$. It suffices to show the convergence of a subsequence. We can extract a subsequence and name it as $\{\{B_{n,m}\}_n+L_w\}_{m}$ itself such that
		$$\tilde{q}_w(\{B_{n,m+1}-B_{n,m}\}_n+L_w)\leq 2^{-m};\, m=1,2,3,\ldots$$
		Then,
		$$\tilde{q}_w(\{B_{n,m+i}-B_{n,m}\}_n+L_w)=q_w(\{B_{n,m+i}-B_{n,m}\}_n)\leq 2^{(1-m)};\, i=1,2,3,\ldots$$
		Also we can construct two matrix sequences $\{R_{n,i}^m\}_n$ and $\{N_{n,i}^m\}_n$ such that 
		$$B_{n,m+i}-B_{n,m}=R_{n,i}^m+N_{n,i}^m,$$
		where $\|N_{n,i}^{m}\|\leq2^{-(m-1)}$ and $\displaystyle \limsup_{n\to \infty} \frac{rank(R_{n,i}^m)}{n} =0$.

		We can find a strictly increasing sequence of positive integers$\{n_{i,m}\}_i$ such that for all $n\geq n_{i,m}$,$\frac{rank(R_{n,i}^m)}{n}<\frac{1}{i}.$
		Also we choose $\{\{n_{i,m}\}_m\}_i$ such that
		\begin{equation}\label{inequality}
			n_{i,m+1}>n_{i+1,m}.
		\end{equation}
		This inequality helps us to obtain the required estimate. Since $n_{i,m+1}>n_{i+1,m}>n_{i,m}$, for a fixed $i$, $\{n_{i,m}\}_m$ is an increasing sequence.\\
		Now consider  $\{n_{2,m}\}_m$ and construct a matrix sequence $\{A_n\}_n$ in a such a way that $A_n=B_{n,j+1}$, whenever $n_{2,j-1}\leq n<n_{2,j}$.\\
		Consider $A_n-B_{n,m}$; for $n_{2,m+i-1}\leq n<n_{2,m+i},$
		$$A_n-B_{n,m}=B_{n,m+i+1}-B_{n,m}=R_{n,i+1}^{m}+N_{n,i+1},$$
		where $\|N_{n,i+1}^m\|\leq 2^{(1-m)}$ and $\displaystyle \frac{rank(R_{n,i+1}^m)}{n}<\frac{1}{i+1}$, for all $n\geq n_{i+1,m}.$\\
		Here by inequality \ref{inequality}, $n\geq n_{2,m+i-1 }>n_{i+1,m}$, then
		$q_w(\{A_n\}_n-\{B_{n,m}\}_n)<2^{(1-m)}$.\\
		Hence $\displaystyle \lim_{m\to \infty} \tilde{q}_w(\{A_n\}_n-\{B_{n,m}\}_n+L_w)=\lim_{m\to \infty} {q}_w(\{A_n\}_n-\{B_{n,m}\}_n) =0$.
		
		Thus, $\tilde{\mathcal{A}}_w$ and $\tilde{\mathcal{A}}_{w^p}$ are Banach spaces.  $\tilde{\mathcal{A}}_w$ form a $C^*$ algbera with usual complex conjugate transpose of the matrix as the involution, that is $\{A_n\}_n^*=\{A_n^*\}_n$.
	\end{proof}
	\sloppy The convergence of a sequence $\{\{B_{n,m}\}_m\}_n$ to $\{A_n\}_n$ in the topology induced by $q_w$ and $q_{w^p}$ are denoted by  $\{\{B_{n,m}\}_n\}_m\xrightarrow{q_w}\{A_n\}_n$ and $ \{\{B_{n,m}\}_n\}_m\xrightarrow{q_{w^p}}\{A_n\}_n$ respectively.
	
	Now we recall a lemma from \cite{stefano} that is useful to obtain the relation between a.c.s. convergence and the convergence with respect to $q_w$ and $ q_{w^p}$.
	\begin{lemma}(Theorem 4.1 \cite{stefano})\label{lemma1}
		Let $\{A_n\}_n$ be a matrix sequence and let $\{\{B_{n,m}\}_n\}_m$ be a sequence of matrix sequences. Then the following conditions are equivalent
		\begin{enumerate}
			\item  $\{\{B_{n,m}\}_n\}_m$ is an a.c.s. for $\{A_n\}_n$
			\item $p(\{A_n-B_{n,m}\}_n)\to 0 \,\,\,as\,\,\, m\to \infty$
		\end{enumerate}
	\end{lemma} 
	The next theorem gives the comparison of these convergence notions.
	\begin{theorem}\label{thm2}
		$ q_{w}$ convergence $\Longrightarrow$ $q_{w^p}$ convergence $\Longrightarrow$ a.c.s. convergence.
	\end{theorem}
	\begin{proof}
		The	proof is immediate from the definition and Lemma \ref{lemma1}. 
	\end{proof}
	\begin{remark}\normalfont
		Reverse implication is false, as we shall see in the following example. Also $q_{w^p}$ convergence implies $q_{w^q}$ convergence if $1\leq p<q<\infty$.
	\end{remark}

	Now we give an example of a sequence of matrix sequences for which the converse of Theorem \ref{thm2} is false.
	
	\begin{example}\normalfont
		Let $B_{n,m}$ be the diagonal matrix with its first $\lfloor \frac{n}{m}\rfloor$ diagonal entries  $1$ and others  $0$.
		\begin{equation*}
			\begin{aligned}
				p(\{B_{n,m}\}_n)&=\inf\left\{\limsup_{n\to \infty}\left\{\frac{rank\, R}{n}+\|N\|:\,R+N=B_{n,m}\right\}\right\}\\
				&\leq \limsup_{n\to \infty} \frac{rank \,B_{n,m}}{n}\\
				&=\frac{\lfloor\frac{n}{m}\rfloor}{n}\\
				&\leq\frac{1}{m}
			\end{aligned}
		\end{equation*}
		Then by lemma \ref{lemma1}, $\{\{B_{n,m}\}_n\}_m \xrightarrow{a.c.s} \{O_n\}_n$.
		But $q_w(\{B_{n,m}\}_n)$ is 1 for all $m$. Hence $\{\{B_{n,m}\}_n\}_m$ does not converge to $\{O_n\}_n$ in $\tilde{\mathcal{A}}_w$.
	\end{example} 
	\section{Main Results: GLT sequences and $L^{p}$ spaces}
	In this section, we prove our main result, Theorem \ref{isometry}. Recall the definition of GLT sequences (see definition \ref{GLT defn}). 
	Let $\mathcal{G}^\infty$ and $\mathcal{G}^p$ be the spaces defined as follows.
	$$\mathcal{G}^\infty=\{\{A_n\}_n\in\mathcal{A}_w: \{A_n\}_n\sim_{GLT}f\},\,\,\,\mathcal{G}^p=\{\{A_n\}_n\in\mathcal{A}_{w^p}: \{A_n\}_n\sim_{GLT}f\}.$$
	Let $$Z=\{\{A_n\}_n\in\mathcal{G}^{\infty}: q_w(\{A_n\}_n)=0\},$$  $$Z^{p}=\{\{A_n\}_n\in\mathcal{G}^{p}: q_{w^p}(\{A_n\}_n)=0\},$$
	and $\tilde{\mathcal{G}}^\infty=\mathcal{G}^{\infty}/Z$, $\tilde{\mathcal{G}^p}=\mathcal{G}^{p}/Z^p$ be the quotient spaces  of $\mathcal{G}^{\infty}$ and $\mathcal{G}^p$ respectively.\\
	Following is an example of a matrix sequence that belongs to $\tilde{\mathcal{G}}^1$ but not to $\tilde{\mathcal{G}}^2$.
	\begin{example}\normalfont
		Let $a:[0,1]\to \mathbb{R}$ be a function defined by 
		$$a(x)=\left\{\begin{array}{ll}
			\frac{1}{\sqrt{x}}& \mbox{if $0 < x \leq 1$}\\
			0 & \mbox{if $x=0.$}
		\end{array}\right. ,$$ $g:[-\pi,\pi]\to \mathbb{C}$ is constant function 1 and $\{A_n\}_n$ be the matrix sequence given by
		$$A_n=\begin{pmatrix}
			\sqrt{n}&0&0&\cdots&0\\
			0&\sqrt{\frac{n}{2}}&0&\cdots&0\\
			0&0&\sqrt{\frac{n}{3}}&\cdots&0\\
			\vdots&\vdots&\vdots&\ddots&\vdots\\
			0&0&0&\cdots&1
			
		\end{pmatrix}.$$
		The matrix sequence $\{A_n\}_n$ belongs to $\tilde{\mathcal{G}}^1 $ but not $\tilde{\mathcal{G}}^2 $ and its symbol is the function $f:[0,1]\times[-\pi,\pi]\to \mathbb{R}$ defined by $f=a\otimes g$.
	\end{example} 
	We recall some algebraic properties of the space of GLT sequences from \cite{stefano};
	\begin{enumerate}
		\item Suppose that $\{A_n\}_n\sim_{GLT}f$ and $\{B_n\}_n\sim_{GLT}g$. Then,\label{tog1}
		\begin{enumerate}
			\item $\{A_{n}^{*}\}_n\sim_{GLT}\bar{f}$,
			\item$\{\alpha A_n+\beta B_n\}_n\sim_{GLT}\alpha f+\beta g$, for all $\alpha, \beta \in \mathbb{C},$
			\item $\{A_nB_n\}_n\sim_{GLT} fg$.
			\item if $\{A_n\}_n\sim_{GLT} h$ then $f=h$ a.e
		\end{enumerate}
		\item For all measurable function $f$ defined on $D:=[0,1]\times [-\pi,\pi]$, there exists a matrix sequence $\{A_n\}_n$ such that $\{A_n\}_n\sim_{GLT}f$.\label{tog2}
	\end{enumerate}
	Define a function $\phi_p:\tilde{\mathcal{G}}^p\to L^{p}(D), 1\leq p \leq \infty,$ such that whenever \,$\{A_n\}_n\sim_{GLT}f$,
	\begin{equation*}
		\begin{aligned}
			\phi_{p}(\{A_n\}_n)&=f&\text{if $p=\infty$}\\
			&=(2\pi)^{\frac{1}{p}}f&\text{if $p\neq \infty$.}	
		\end{aligned}
	\end{equation*}
	$\phi_p$ is well defined by the property (d) of (\ref{tog1}).
	Now we are in a position to prove our main result Theorem \ref{isometry}. In fact we prove that 	$\phi_p$ is an isometric isomorphism between  $\tilde{\mathcal{G}}^p$ and $ L^{p}(D), 1\leq p \leq \infty.$ This is a consequence of the above listed properties and the following couple of lemmas.
	\begin{lemma}\label{eon1}
		Let $\{A_n\}_n\in \mathcal{A}_w$ and $\{A_n\}_n\sim_{\sigma}f$, then $q_w(\{A_n\}_n)=\|f\|_{\infty}$.
	\end{lemma}
	\begin{proof}
		Suppose $\|f\|_{\infty}= \underset{x\in D}{ess\,sup}\,|f(x)|=l$.
		
		Then by the definition of essential supremum, for any $\epsilon>0$, 
		$$\mu\{x:|f(x)|\geq l+\epsilon\}=0,$$
		where $\mu$ is the Lebesgue measure. Also it is given that $\{A_n\}_n\sim_{\sigma}f$, then 
		$$\lim_{n\to \infty}\frac{1}{n}\sum_{i=1}^{n}F(\sigma_i(A_n))=\frac{1}{2\pi}\int_{D}F(|f(x)|)dx$$
		for every $F:\mathbb{R}\to \mathbb{C}$ continuous function with compact support and singular values $\sigma_i(A_n)$ arranged in non increasing order; $\sigma_1(A_n)\geq \sigma_2(A_n)\geq\cdots\geq\sigma_n(A_n)$.
		Consider a real valued continuous function $F$ with compact support, such that $\chi_{[\epsilon,l+2\epsilon]}\geq F(x)\geq\chi_{[0,l+\epsilon]}$, then
		$$\lim_{n\to \infty}\frac{1}{n}\sum_{i=1}^{n}F(\sigma_i(A_n))\leq \liminf_{n\to \infty}\frac{1}{n}\#(\sigma_i(A_n)\leq l+2\epsilon) \,\,\text{and}$$
		$$\int_{D}F(|f(x)|)dx\geq \mu\{x:|f(x)|\leq l+\epsilon\}.$$
		Then,
		$$\liminf_{n\to \infty}\frac{1}{n}\#(\sigma_i(A_n)\leq l+2\epsilon)\geq \frac{1}{2\pi}\mu\{x:|f(x)|\leq l+\epsilon\},$$
		$$\liminf_{n\to \infty}\frac{1}{n}\{n-(\#\sigma_i(A_n)> l +2\epsilon)\}\geq \frac{1}{2\pi}(2\pi-\mu\{x:|f(x)|> l+\epsilon\}),$$
		$$1-\limsup_{n\to \infty}\frac{1}{n}(\#\sigma_i(A_n)> l +2\epsilon)\geq 1-\frac{1}{2\pi}\mu\{x:|f(x)|> l+\epsilon\},$$
		$$\limsup_{n\to \infty}\frac{1}{n}\#(\sigma_i(A_n)> l +2\epsilon)\leq \frac{1}{2\pi}\mu\{x:|f(x)|> l+\epsilon\} =0.$$
		
		By Theorem \ref{sigma}, $q_w(\{A_n\}_n)\leq l+2\epsilon.$
		Thus, $q_w(\{A_n\}_n)\leq \|f\|_{\infty}$.\\
		To prove the other inequality, suppose $q_w(\{A_n\})=k$. By Theorem \ref{sigma}, for $\epsilon>0$, $\underset{n\to \infty}{\lim}\frac{\#(\sigma(A_n)>k+\epsilon)}{n}=0.$
		
		Given that $\{A_n\}_n\sim_{\sigma}f$. Then,
		$$\lim_{n\to \infty}\frac{1}{n}\sum_{i=1}^{n}F(\sigma_i(A_n))=\frac{1}{2\pi}\int_{D}F(|f(x)|)dx.$$
		for every $F:\mathbb{R}\to \mathbb{C}$  continuous function with compact support. Consider such a function $F:\mathbb{R}\to \mathbb{C}$ such that $\chi_{[-\epsilon,k+2\epsilon]}\geq F\geq\chi_{[0,k+\epsilon]}$.
		$$\liminf_{n\to \infty}\frac{\#(\sigma(A_n)\leq k+\epsilon)}{n}\leq\lim_{n\to \infty}\frac{1}{n}\sum_{i=1}^{n}F(\sigma_i(A_n)),$$
		$$\int_{D}F(|f(x)|)dx\leq \mu\{x:|f(x)|\leq k+2\epsilon\},$$
		$$\liminf_{n\to \infty}\frac{1}{n}\#(\sigma_i(A_n)\leq k +\epsilon)\leq \frac{1}{2\pi}\mu\{x:|f(x)|\leq k+2\epsilon\},$$
		$$\frac{1}{2\pi}\mu\{x:|f(x)|> k+2\epsilon\} \leq\limsup_{n\to \infty}\frac{1}{n}\#(\sigma_i(A_n)> k +\epsilon)=0. $$
		Thus, $\|f\|_\infty\leq k+2\epsilon$ and $\|f\|_\infty\leq q_w(\{A_n\}_n)$.
	\end{proof}
	\begin{cor}
		Let $\{A_n\}_n\sim_\sigma f$. Then $\{A_n\}_n\in\mathcal{A}_w$ if and only if $f\in L^{\infty}(D)$.
	\end{cor}
	\begin{lemma}\label{eon2}
		Let $\{A_n\}_n\in \mathcal{A}_{w^p}$ and $\{A_n\}_n\sim_{\sigma}f$. Then $q_{w^p}(\{A_n\}_n)=\frac{1}{(2\pi)^{(1/p)}}\|f\|_p, 1\leq p<\infty$.
	\end{lemma}
	\begin{proof}
		The proof  is similar to Lemma \ref{eon1}. 
	\end{proof}
	\begin{proof}[Proof of Theorem \ref{isometry}:]	$\phi_p$ is an injective *-homomorphism of Banach spaces, which can be readily followed from (b) of property (\ref{tog1}) of GLT. The surjectivity follows from  the definition of $\tilde{\mathcal{G}}^p$ and property (\ref{tog2}) of GLT. Hence it is a *-isomorphism. From Lemma \ref{eon1} and Lemma \ref{eon2}, it follows that $\phi_{p}$ is an isometry. In particular, $\phi_\infty$ is a $C^*$- isomorphism.
	\end{proof}

	Theorem \ref{isometry} yields a natural isometry between the spaces $\tilde{\mathcal{G}}^p$ and $L^p(D)$, for $1\leq p \leq \infty $, which is analogous to the isometry identified in \cite{barbarino} between space of GLT sequences and space of measurable functions. Notice that in \cite{barbarino}, the author derived a metric space isometry. Here we achieved a Banach space isometric isomorphism.
	\section{Korovkin-Type Theorem}
	P. P Korovkin proved a classical approximation theorem in 1953, which unified several approximation process.  Korovkin-type theorems in the setting of Toeplitz operators acting on Hardy spaces and Fock spaces were obtained in \cite{rahul,vbkpreconditioner}.  Type 2 strong/weak cluster sense convergence was considered there.  Here we obtain an analogous result for GLT sequences. 
	
	Consider $M_n(\mathbb{C})$, with Frobenius norm  induced from the inner product $\langle A,B\rangle = $ trace$(B^*A)$. Let $\{U_n\}_n$ be sequence of unitary matrices such that each $U_n$ is of order $n$. For each $n$, define the subalgebra $M_{U_n}$ of $M_n(\mathbb{C})$  as
	$$M_{U_n}=\{A\in M_n(\mathbb{C}) : U_{n}^*AU_n \text{ is diagonal }\}.$$
	$M_{U_n}$ is a closed subspace of $M_n(\mathbb{C})$. We denote the orthogonal projection of $M_n(\mathbb{C})$ onto $M_{U_n}$ by $P_{U_n}(\cdot)$. It is known that $\|P_{U_n}(\cdot)\|=1$ when we consider $M_n(\mathbb{C})$ as a Banach space under usual operator norm (see \cite{vbkpreconditioner} for details). For $A\in M_n(\mathbb{C})$, $P_{U_n}(A)$ is called a preconditioner for $A$.
	
	Preconditioners play a crucial role in solving linear systems by iterative techniques.  They help to increase the convergence rate of iteration. For instance, consider the linear system with Toeplitz structure,
	$$T_n(f)x=b_n.$$
	For a fixed $f$, we can consider a sequence of Toeplitz matrices $\{T_n(f)\}_n$. If we can find a sequence of matrices $\{C_n(f)\}_n$ such that $\{C_n(f)-T_n(f)\}_n$ converges to $\{O_n\}_n$ in Type 2 strong/weak cluster sense, $\{C_n(f)\}_n$ can be considered as an efficient preconditioner \cite{rahulsurvey}. In this case, the eigenvalues of $C_n(f)^{-1}T_n(f)$ will be clustered at 1. This will help to improve the stability of the corresponding linear system.  In \cite{chan1992circulant} R. H Chan and M. C yeung proved that when $U_n=F_n$, the Fourier matrix of order $n$ and $f$ is a continuous function, then $\{P_{U_n}(T_n(f))-T_n(f)\}_n$ converges to $\{O_n\}_n$ in Type 2 strong cluster sense (corresponding preconditioners are known as circulant preconditioner). Depending on the choice of $U_n$, we can obtain other efficient preconditioners such as Hartley \cite{jin1994hartley}, Tau \cite{serra1999superlinear} etc. for Toeplitz matrices.
	
	Since linear systems involving GLT sequences appear at various situations, finding efficient preconditioners for GLT sequences is also an  important problem.\\
	Following is an example of an efficient preconditioner for GLT sequences.
	
	Consider a LT sequence $\{A_n\}_n\sim_{LT}a\otimes f$,  where $a$ is a Riemann integrable function on $[0,1]$ and $f$ is a continuous function on $[-\pi ,\pi]$. We give an example of preconditioner for this LT sequence and also  obtain a preconditioner for a GLT sequence. Let
	$$U_n=\begin{pmatrix}
		F_{\lfloor\frac{n}{m}\rfloor}&0&0&\cdots&0\\
		0&F_{\lfloor\frac{n}{m}\rfloor}&0&\cdots&0\\
		0&0&F_{\lfloor\frac{n}{m}\rfloor}&\cdots&0\\
		\vdots&\vdots&\vdots&\ddots&\vdots\\
		0&0&0&\cdots&F_{n(mod\, m)}	
	\end{pmatrix},$$
	where $F_n=\left(\frac{1}{\sqrt{n}}e^{\frac{2\pi ijl}{n}}\right)_{j,l=0}^{n-1}$ is the Fourier matrix of order $n$. Consider 
	$$LT_{n}^{m}(a,f)=D_m(a)\otimes T_{\lfloor\frac{n}{m}\rfloor}(f)\oplus O_{n(mod\, m)}.$$ 
	We can construct a matrix sequence $\{\tilde{A}_n\}_n$  which is an a.c.s limit for the sequence $\{\{LT_{n}^{m}(a,f)\}_n\}_m$ such that $\tilde{A}_n=LT_{n}^{m}(a,f)$ for some $m$ and $n\geq m^2$.  Now consider \\
	\begin{equation*}	
		\quad\begin{aligned}
			P_{U_n}(\tilde{A}_n)-\tilde{A}_n=
			\begin{pmatrix}
				a(\frac{1}{m})P_{F_k}(T_k(f))&0&0&\cdots&0\\
				0&	a(\frac{2}{m})P_{F_k}(T_k(f))&0&\cdots&0\\
				\vdots&\vdots&\ddots&\cdots&\vdots\\
				0&0&\cdots&	a(1)P_{F_k}T_k(f))&0\\
				0&0&\cdots&0&O_{n(mod\, m)}
			\end{pmatrix}\\
			-\begin{pmatrix}
				a(\frac{1}{m})T_k(f)&0&0&\cdots&0\\
				0&	a(\frac{2}{m})T_k(f)&0&\cdots&0\\
				\vdots&\vdots&\ddots&\cdots&\vdots\\
				0&0&\cdots&	a(1)T_k(f)&0\\
				0&0&\cdots&0&O_{n(mod\, m)}
			\end{pmatrix},
		\end{aligned}
	\end{equation*}
	where $k={\lfloor\frac{n}{m}\rfloor}$. Since $\{P_{F_k}(T_k(f))-T_k(f)\}_n$ converges to $\{O_n\}_n$ in Type 2 strong cluster sense, we can show that $\{P_{U_n}(\tilde{A}_n)-\tilde{A}_n\}_n$ converges to $\{O_n\}_n$ in Type 2 weak cluster sense. Since $q_w(\{A_n-\tilde{A}_n\}_n)=0, \{P_{U_n}(\tilde{A}_n)-A_n\}_n$ converges to $\{O_n\}_n$ in Type 2 weak cluster sense.
	\begin{remark}
		Note that $\{A_n\}_n$ in the above example need not be norm bounded.	If $\{A_n\}_n$ is norm bounded, then $\{P_{U_n}(A_n)-A_n\}_n$ converges to $\{O_n\}_n$ in Type 2 weak cluster sense (see Lemma \ref{qw pre}).
	\end{remark}
	Now consider a GLT sequence $\{B_n\}_n$ with symbol $\kappa$, such that  $\sum \limits_{i=1}^{k_m} a_{i,m} \otimes f_{i,m}$ converges to $\kappa$ in essential supremum norm (or in measure), where each $a_{i,m}$ is a Riemann integrable function on $[0,1]$ and $f_{im}$ is a continuous function on $[-\pi,\pi]$. Then $\{\{\sum\limits_{i=1}^{k_m}P_{U_n}(D_n(a_{i,m})T_n(f_{i,m}))\}_n\}_m$  converges to $\{B_n\}_n$ in $\tilde{\mathcal{G}}^{\infty} $ (or a.c.s for $\{B_n\}_n )$. Then we can construct a sequence $\{\hat{A}_n\}_n$ as in the proof of Theorem \ref{completeness} such that $\hat{A}_n=P_{U_n}(\sum\limits_{i=1}^{k_m}D_n(a_{i,m})T_n(f_{i,m}))$ for some $m$  (depends on $n$) and $q_w(\{\hat{A}_n\}_n-\{B_n\}_n)=0.$ Thus $\{\hat{A}_n\}_n$ is a good preconditioner for $\{B_n\}_n$.
	
	Since the convergence of $\{P_{U_n}(T_n(f))-T_n(f)\}_n$ to $\{O_n\}_n$  in Type 2 strong/weak cluster sense leads to efficient preconditioners, it is important to know when does this convergence holds. The Korovkin-type theorems obtained in \cite{vbkpreconditioner} reduces this task into a finite subset of the class of symbols. Here we obtain a similar result in the setting of GLT sequences. First, we prove a couple of lemmas.
	
	\begin{lemma}\label{F type 2}
		Suppose $\{A_n\}_n$ is a sequence of matrices of growing order such that $\|A_n\|\leq M<\infty$, for some $M>0$. Then $q_w(\{A_n\}_n)=0$ if and only if $\|A_n\|_{F}^{2}=o(n)$.
	\end{lemma}
	\begin{proof}
		If $\|A_n\|_{F}^2=o(n)$, then by Lemma \ref{frobenius}, $\{A_n\}_n$ converges to $\{O_n\}_n$ in Type 2 weak cluster sense. Then $q_w(\{A_n\}_n)=0$.\\
		Conversely assume that $q_w(\{A_n\}_n)=0$. Then $\{A_n\}_n$ converges to $\{O_n\}_n$ in Type 2 weak cluster sense. Using Theorem \ref{thm1}, we get  $\displaystyle \lim_{n\to \infty}P(A_n)=0.$ Also by Theorem \ref{p(a)}, we have 
		
		$$P(A_n)=\min_{i=1,2,\ldots,n}\left\{\frac{i}{n}+\sigma_{i+1}(A_n)\right\}.$$
		Since $\displaystyle \lim_{n\to \infty}P(A_n)=0,$ for $\epsilon >0,$ there exists a positive integer $n_\epsilon$ such that for all $n\geq n_\epsilon,\,  P(A_n)<\epsilon.$ Hence we have for $n\geq n_\epsilon$,
		$$\min_{i=1,2,\ldots,n}\left\{\frac{i}{n}+\sigma_{i+1}(A_n)\right\}<\epsilon.$$
		Then there exist a $j$, such that $\frac{j}{n}+\sigma_{j+1}(A_n)<\epsilon.$
		Now $\displaystyle \frac{1}{n}\left(\sum_{i=1}^{j}{\sigma_i}^2\right)< M^2\epsilon$ and $\displaystyle\sum_{i=j+1}^{n}{\sigma_i}^2<{\epsilon}^2(n-j).$ Then,
		$$\|A_n\|_{F}^2=\sum_{i=1}^{n}{\sigma_i}^2(A)<M^2\epsilon +(n-j)\epsilon^2 \text{ and }
		\frac{\|A_n\|_{F}^2}{n}<\frac{M^2\epsilon}{n}+{\epsilon}^2,\,\,\forall n\geq n_\epsilon.$$
		Hence $\|A_n\|_{F}^{2}=o(n).$
	\end{proof}
	\begin{lemma}\label{qw pre}
		Let	$\{A_n\}_n$ be a sequence of matrices such that for each $n, \|A_n\|\leq M <\infty$. If $q_w(\{A_n\}_n)=0,$ then $q_w(\{P_{U_n}(A_n)\}_n)=0.$
	\end{lemma}
	\begin{proof}
		Given $q_w(\{A_n\}_n)=0.$  Then by Lemma \ref{F type 2}, $\|A_n\|_{F}^2=o(n).$ Since  $\|P_{U_n}\|_F=1$, $\|P_{U_n}(A_n)\|_{F}^2\leq \|A_n\|_{F}^2=o(n).$  Thus $q_w(\{P_{U_n}(A_n)\}_n)=0.$
	\end{proof}
	Now we present the Korovkin-type theorem in the setting of GLT sequences. Here we obtain preconditioners for the norm bounded GLT sequences. For arbitrary GLT sequences, see the corollary \ref{corkorvkin}.
	\begin{theorem}\label{korovkin}
		Let $\{f_1,f_2,$\ldots$,f_k\}\subseteq L^{\infty}(D)$ and $\{A_n(f_i)\}_n\sim_{GLT}f_i$ is norm bounded for each $f_i$. Suppose that $\{P_{U_n}(A_n(g))-A_n(g)\}$ converges to $\{O_n\}$ in Type 2 weak cluster sense for $g\in\{f_1,f_2,$\ldots$,f_k,\sum_{i=1}^{k}f_if_i^*\}$ . Then for every $f$ in the $C^*$-algebra generated by $\{f_1,f_2,\ldots,f_k\}$, with $\{A_n(f)\}_n\sim_{GLT}f$ is norm bounded, $\{P_{U_n}(A
		_n(f))-A_n(f)\}$ converges to $\{O_n\}$ in Type 2 weak cluster sense. 
	\end{theorem}
	\begin{proof}
		There is a standard procedure to obtain convergence of $\{P_{U_n}(A_n(f))-A_n(f)\}_n$ with $f$ in the *-algebra generated by $\{f_1,f_2,$\ldots$,f_k\}$. We can follow the same procedure here (see the proof Theorem 3.4 in \cite{rahul}). From the *-algebra to reach $C^*$-algebra (that is the closure in the $C^*$-algebra norm), we proceed as follows.
		
		Let $g$ belongs to $ C^*$-algebra generated by $\{f_1,f_2,\cdots,f_k\}$ and  $\{g_m\}$ be a sequence  which converges to $g$. Then by Theorem \ref{isometry} there exist norm bounded GLT sequences $\{A_n(g_m)\}n$ and $\{A_n(g)\}_n$ corresponding to each $g_m$ and $g$ respectively, such that $\{\{A_n(g_m)\}_n\}_m$ converges to  $\{A_n(g)\}_n$	in $\tilde{\mathcal{G}}^{\infty}$.
		
		Therefore, for $\epsilon >0$, there exists a positive integer $t$ such that $q_w(\{A_n(g_t)-A_n(g)\}_n)<{\epsilon/2}.$ Then there exist two norm bounded sequences $\{R_n\}_{n,t}$ and $\{N_n\}_{n,t}$ such that $$A_n(g_t)-A_n(g)=R_{n,t}+N_{n,t},\quad \lim_{n\to\infty}\frac{rank(R_{n,t})}{n}=0,\quad \|N_{n,t}\|<\epsilon/2.$$
		Now consider
		\begin{equation*}
			\begin{aligned}
				q_w(\{P_{U_n}(A_n(g))-A_n(g)\}_n)&=q_w(\{P_{U_n}(A_n(g))-P_{U_n}(A_n(g_t))\\
				&\quad+P_{U_n}(A_n(g_t))-A_n(g_t)+A_n(g_t)-A_n(g)\}_n)\\
				&\leq q_w(\{P_{U_n}(A_n(g)-A_n(g_t))\}_n)\\
				&\quad+q_w(\{P_{U_n}(A_n(g_t))-A_n(g_t)\}_n)\\
				&\quad+q_w(\{A_n(g_t)-A_n(g)\})_n).
			\end{aligned}
		\end{equation*}
		The second term  on the right hand side is zero and $q_w(\{A_n(g_t)-A_n(g)\})_n)<\epsilon/2.$ Now consider the first term on the right hand side
		\begin{equation*}
			\begin{aligned}
				q_w(\{P_{U_n}(A_n(g)-A_n(g_t))\}_n)&=q_w(\{P_{U_n}(R_{n,t}+N_{n,t})\}_n)\\
				&\leq q_w(\{P_{U_n}(R_{n,t})\}_n)+q_w(\{P_{U_n}(N_{n,t})\}).\\
			\end{aligned}
		\end{equation*}
		Since $q_w(\{R_{n,t}\}_n)=0$, by Lemma \ref{qw pre}, $q_w(\{P_{U_n}(R_{n,t})\}_n)=0.$ Also we know that $\|P_{U_n}\|=1$, then $\|P_{U_n}(N_{n,t})\|\leq \|N_{n,t}\|$ and hence $q_w(\{P_{U_n}(N_{n,t})\}_n)<\epsilon/2$. Thus $q_w(\{P_{U_n}(A_n(g))-A_n(g)\}_n)<\epsilon$. Hence $\{P_{U_n}(A_n(g))-A_n(g)\}_n$ converges to $\{O_n\}_n$ in Type 2 weak cluster sense.
	\end{proof}
	\begin{cor}\label{corkorvkin}
		Under the conditions of Theorem \ref{korovkin}, if  $ \sum_{i=1}^{k_m}g_{i,m}$ converges to $g$ in measure, where each $g_{i,m}$ belong to  $C^*$- algebra generated by $\{f_1,f_2,\ldots,f_k\}$, then we can extract a preconditioner sequence $\{P_{U_n}(A_n(h_i))\}_n$ for $\{A_n(g)\}_n$. 
	\end{cor}
	\begin{proof}
		We have  $ \sum_{i=1}^{k_m}g_{i,m}$ belongs to $ C^*\{f_1,f_2,\cdots,f_m\}$. Hence $\{P_{U_n}(A_n(\sum_{i=1}^{k_m}g_{i,m}))-A_n(\sum_{i=1}^{k_m}g_{i,m})\}_n$ converges to $\{O_n\}_n$ in Type 2 weak cluster sense if $A_n(\sum_{i=1}^{k_m}g_{i,m})$ is of bounded norm and $A_n(\sum_{i=1}^{k_m}g_{i,m})\sim_{GLT}\sum_{i=1}^{k_m}g_{i,m}$. Since $\sum_{i=1}^{k_m}g_{i,m}$ converges to $g$ in measure,  $A_n(\sum_{i=1}^{k_m}g_{i,m})$ and $\{P_{U_n}(A_n(\sum_{i=1}^{k_m}g_{i,m}))$ are   a.c.s for $A_n(g)$. We can construct a sequence of the form $\{P_{U_n}(A_n(h_j))\}_n$ which is an a.c.s limit for $\{P_{U_n}(A_n(\sum_{i=1}^{k_m}g_{i,m}))\}_n$. Hence the result follows.
	\end{proof}
	\begin{remark}
		Theorem \ref{korovkin} need not hold if $\{A_n(f)\}$ is not norm bounded. Let $f\in L^\infty(D)$ and consider a GLT sequence $\{A_n(f)\}_n$ such that $\|A_n(f)\|\leq M<\infty,$ for all $n$. Let, $q_w(\{P_{U_n}(A_n(f)-A_n(f))\}_n=0$ and $\{B_n\}_n$ be another sequence such that $q_w(\{B_n-A_n(f)\}_n)=0.$ But if $\|B_n\|$ is unbounded, then $q_w(\{P_{U_n}(B_n)-B_n\}_n)$ need not be zero. For if, consider $B_n=A_n(f)+Z_n$, where $U_n^*Z_nU_n=(a_{ij})_{i,j=1}^{n}, a_{ij}=1$ for all $1\leq i,j\leq n$. Clearly $q_w(\{P_{U_n}(B_n)-B_n\}_n)\neq0.$ But $q_w(\{P_{U_n}(A_n(f))-B_n\})=0.$ So we can treat $P_{U_n}(A_n(f))$ as a preconditioner for $\{B_n\}_n$. Also note that the function $g$ in corollary \ref{corkorvkin} need not be essentially bounded.	
	\end{remark}
	\section{Concluding Remarks}	
	As we know, the theory of Toeplitz matrix sequences has a rich operator theoretic analogue on the Hardy space via the symbol-function. There are variations of it into Bergman space, Fock space, etc. We expect such versions in the case of GLT sequences also. The development must be through the identification of corresponding symbols.
	The major achievement of this article is that we are able to identify the connection between the space of symbols and the subspaces of GLT sequences. We hope that these identifications will be useful in establishing the operator theoretic analogue of the spectral distributional results of such matrix sequences. Korovkin-type result we obtained in this article makes use of a topology on the space of GLT sequences. The connection with the topologies on $B(\mathbb{H})$ and the topologies introduced in this article would be another interesting point. Obtaining the convergence in eigenvalue clustering as convergence with respect to some topology on $B(\mathbb{H})$ is the main goal in our future research.
	\section*{Acknowledgments}
	V. B. Kiran Kumar is thankful to KSCSTE, Kerala for financial support through the KSYSA-Research Grant. Rahul Rajan is thankful to Cochin University of Science and Technology for the financial support under the University post doctoral fellowship. N. S. SarathKumar is supported by CSIR-
	JRF.

\end{document}